\documentclass[10pt]{article}

\usepackage{palatino}
\usepackage{amsmath}
\usepackage{amsthm}
\usepackage{amssymb}
\usepackage{latexsym}
\usepackage{titlefoot}
\usepackage[small]{titlesec}
\usepackage[small,it]{caption}
\usepackage{mathabx}
\usepackage{makecell}
\usepackage{tikz}
\usepackage{multirow}
\usepackage{enumerate}
\usepackage{cancel}
\usepackage{bm}

\usepackage{color}
\definecolor{lightgray}{rgb}{0.8, 0.8, 0.8}
\definecolor{darkgray}{rgb}{0.7, 0.7, 0.7}
\definecolor{darkblue}{rgb}{0, 0, .4}

\usepackage[bookmarks]{hyperref}
\hypersetup{
        colorlinks=true,
        linkcolor=darkblue,
        anchorcolor=darkblue,
        citecolor=darkblue,
        urlcolor=darkblue,
        pdfpagemode=UseThumbs,
        pdftitle={On the effective and automatic enumeration of polynomial permutation classes},
        pdfsubject={Combinatorics},
        pdfauthor={Homberger and Vatter},
}

\newcounter{todocounter}

\theoremstyle{plain}
\newtheorem{theorem}{Theorem}[section]
\newtheorem{proposition}[theorem]{Proposition}

\theoremstyle{definition}

\newenvironment{proofsketch}{\noindent{\it Sketch of proof.}}{\qed\bigskip}

\makeatletter
\newfont{\footsc}{cmcsc10 at 8truept}
\newfont{\footbf}{cmbx10 at 8truept}
\newfont{\footrm}{cmr10 at 10truept}
\renewenvironment{abstract}%
                {
                  \begin{list}{}%
                     {\setlength{\rightmargin}{1in}%
                      \setlength{\leftmargin}{1in}}%
                   \item[]\ignorespaces\begin{small}}%
                 {\end{small}\unskip\end{list}}

\newcommand{\C}{\mathcal{C}}
\newcommand{\V}{\mathcal{V}}
\newcommand{\W}{\mathcal{W}}

\newcommand{\vect}[1]{\left\langle #1 \right\rangle}

\newcommand{\OEISlink}[1]{\href{http://oeis.org/#1}{#1}}
\newcommand{\OEISref}{\href{http://oeis.org/}{OEIS}~\cite{sloane:the-on-line-enc:}}

\newcommand{\Grid}{\operatorname{Grid}}

\newcommand{\st}{\::\:}

\newcommand{\p}[1]{#1^+}
\newcommand{\m}[1]{#1^-}
\renewcommand{\d}[1]{#1^{\bullet}}
\newcommand\diag[4]{%
  \multicolumn{1}{p{#2}|}{\hskip-\tabcolsep
  $\vcenter{\begin{tikzpicture}[baseline=0,anchor=south west,inner sep=#1]
  \path[use as bounding box] (0,0) rectangle (#2+2\tabcolsep,\baselineskip);
  \node[minimum width={#2+2\tabcolsep},minimum height=\baselineskip+\extrarowheight] (box) {};
  \draw (0.25ex,\baselineskip+0.05ex) -- (box.south east);
  \node[overlay] at (0.25ex,-0.05ex) {$k$};
  \node[overlay] at (2.25ex,1ex) {$n$};
  \end{tikzpicture}}$\hskip-\tabcolsep}
}
\newcommand{\diagnk}{\diag{0.1em}{0.1cm}{$k$}{$n$}}
\newcommand{\nc}[1]{{n \choose #1}}

\datefoot{\today}
\amssubj{Primary: 05A15; Secondary: 05A05}
\newpagestyle{main}[\small]{
        \headrule
        \sethead[\usepage][][]
        {\sc Polynomial Permutation Classes}{}{\usepage}}

\title{\sc On the Effective and Automatic Enumeration of Polynomial Permutation Classes}
\author{% And here it begins
	\begin{tabular}{cc}
        Cheyne Homberger&Vincent Vatter\footnote{Vatter's research was sponsored by the National Security Agency under Grant Number H98230-12-1-0207 and the National Science Foundation under Grant Number DMS-1301692.  The United States Government is authorized to reproduce and distribute reprints not-withstanding any copyright notation herein.}\\
		{\small Department of Mathematics \& Statistics}&{\small Department of Mathematics}\\[-3pt]
		{\small University of Maryland, Baltimore County}&{\small University of Florida}\\[-3pt]
		{\small Baltimore, Maryland, USA}&{\small Gainesville, Florida, USA}\\[-10pt]
	\end{tabular}
}

\titleformat{\section}
        {\large\sc}
        {\thesection.}{1em}{}   
\titleformat{\subsection}
        {\normalsize\bf}
        {\thesubsection.}{1em}{\vspace{-1ex}}[\vspace{-1ex}]

\date{}

\usepackage{inconsolata}
\usepackage[ruled, figure]{algorithm2e}

\begin{document}
\maketitle

\pagestyle{main}

\begin{abstract}
We describe an algorithm, implemented in Python, which can enumerate any permutation class with polynomial enumeration from a structural description of the class. In particular, this allows us to find formulas for the number of permutations of length $n$ which can be obtained by a finite number of block sorting operations (e.g., reversals, block transpositions, cut-and-paste moves).
\end{abstract}

\section{Introduction}\label{sec-poly-intro}

The Fibonacci Dichotomy of Kaiser and Klazar~\cite{kaiser:on-growth-rates:} was one of the first general results on the enumeration of permutation classes. It states that if there are fewer permutations of length $n$ in a class than the $n$th Fibonacci number, for any $n$, then the enumeration of the class is given by a polynomial for sufficiently large $n$. Since the Fibonacci Dichotomy was established for permutation classes, Balogh, Bollob\'as, and Morris~\cite{balogh:hereditary-prop:ordgraphs} showed that it extends to the (more general) context of ordered graphs, while other proofs of the Fibonacci Dichotomy for permutations have been given by Huczynska and Vatter~\cite{huczynska:grid-classes-an:} and Albert, Atkinson, and Brignall~\cite{albert:permutation-cla:}.

%Recent research on permutation classes has made use of structural definitions which capture the geometric structure of the classes themselves, often in the form of generalized and geometric grid classes.
%The geometric structure of polynomial classes is easy to describe, but a method for actually enumerating such a class from such a description has proved elusive.
While much of the focus on this strand of research has shifted to the consideration of larger classes (see Bollob\'as~\cite{bollobas:hereditary-and-:BCC:}, Klazar~\cite{klazar:overview-of-som:}, and Vatter~\cite{vatter:permutation-cla:} for surveys), we return to consider two open questions about polynomial classes. 

\begin{itemize}
\item {\bf Question 1.1.} Given a structural description of a polynomial permutation class, how can we enumerate it?
\addtocounter{theorem}{1}
\item {\bf Question 1.2.} Which polynomials occur as enumerations of polynomial classes?
\addtocounter{theorem}{1}
\end{itemize}

We view a satisfactory answer to Question 1.1 as a prerequisite for the investigation of Question 1.2, and thus our focus in this paper is on enumerating polynomial classes from a structural description. Our answer to Question 1.1 also has applications to the study of genome rearrangements, as discussed in Section~\ref{sec-genome-rearrangement}. In particular, the algorithm can be applied to the problem of evolutionary distance, which investigates the number of genomes of fixed mutation distance from the identity.

The permutation $\pi$ of length $n$ contains the permutation $\sigma$ of length $k$ (written $\sigma\le\pi$) if $\pi$ has a subsequence of length $k$ which is order isomorphic to $\sigma$. For example, $\pi=391867452$ (written in list, or one-line notation) contains $\sigma=51342$, as can be seen by considering the subsequence $91672$ ($=\pi(2)\pi(3)\pi(5)\pi(6)\pi(9)$). A \emph{permutation class}, or simply \emph{class}, is a downset in this subpermutation order; thus if $\C$ is a class, $\pi\in\C$, and $\sigma\le\pi$, then $\sigma\in\C$.

While there are many ways to specify a class, two are particularly relevant to this problem. One is by the \emph{basis} of the class, the minimal permutations \emph{not} in the class. One may also specify a polynomial class by providing some structural description. We adopt this structural approach to the specification of classes.

We must first formalize the notion of the structure of polynomial classes. Following Albert and Atkinson~\cite{albert:simple-permutat:}, an \emph{interval} in a permutation is a sequence of contiguous entries whose values form an interval of natural numbers. A \emph{monotone interval} is an interval in which the entries are monotone (increasing or decreasing). Given a permutation $\sigma$ of length $m$ and nonempty permutations $\alpha_1,\dots,\alpha_m$, the \emph{inflation} of $\sigma$ by $\alpha_1,\dots,\alpha_m$ is the permutation $\pi=\sigma[\alpha_1,\dots,\alpha_m]$ obtained by replacing each entry $\sigma(i)$ by an interval that is order isomorphic to $\alpha_i$, while maintaining the relative order of the intervals themselves. For example,
\[
	3142[1,321,1,12]=6\ 321\ 7\ 45.
\]
Going against traditional conventions, in this work we \emph{allow inflations by the empty permutation} unless specifically forbidden.

The polynomial classes are very special cases of geometric grid classes~\cite{albert:geometric-grid-:}, and they can therefore be described, roughly, as classes for which the entries of every member of the class can be partitioned into a finite number of monotone intervals, which are related to each other in one of a finite number of ways. To describe this more concretely, let us say that a \emph{peg permutation} is a permutation where each entry is decorated with a $+$, $-$, or $\bullet$, such as
\[
\tilde{\rho}=\d{3}\m{1}\d{4}\p{2}
\]
As demonstrated above, we decorate peg permutations with tildes; in this context, $\rho$ denotes for us the underlying (non-pegged) permutation, $3142$ in this example.

The \emph{grid class} of the peg permutation $\tilde{\rho}$, denoted $\Grid(\tilde{\rho})$, is the set of all permutations which may be obtained by inflating $\rho$ by monotone intervals of type determined by the signs of $\tilde{\rho}$: $\rho(i)$ may be inflated by an increasing (resp., decreasing) interval if $\tilde{\rho}(i)$ is decorated with a $+$ (resp., $-$) while it may only be inflated by a single entry (or the empty permutation) if $\tilde{\rho}(i)$ is dotted. Thus $\pi\in\Grid(\tilde{\rho})$ if its entries can be partitioned into monotone intervals which are compatible with $\tilde{\rho}$; we refer to this as a \emph{$\tilde{\rho}$-partition} of $\pi$.

Given a set $\tilde{G}$ of peg permutations, we denote the union of their corresponding grid classes by
\[
\Grid(\tilde{G})=\bigcup_{\tilde{\rho}\in\tilde{G}} \Grid(\tilde{\rho}).
\]
As the next result shows, our goal is to enumerate such classes.

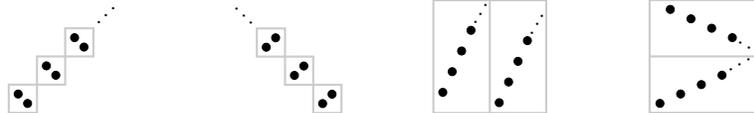
\begin{figure}
\begin{center}
\begin{tabular}{ccccccc}
	\begin{tikzpicture}[scale=0.125]
		  % loops over the 3 boxes
		  \foreach \i in {1, 4, 7}{
		    \draw[thick, color=lightgray, line cap=round] (\i,\i)--(\i, \i+3)--(\i+3, \i+3)--(\i+3,\i)--cycle;
		    \draw[fill = black] (\i+2, \i+1) circle (12pt);
		    \draw[fill = black] (\i+1, \i+2) circle (12pt);
		  }
	      \node [rotate=45] at (11.5,11.5) {{\footnotesize $\dots$}};
	  \end{tikzpicture}
&\quad\quad&
	  \begin{tikzpicture}[scale=0.125]
		  \foreach \i in {1, 4, 7}{
		    \draw[thick, color=lightgray, line cap=round] (\i,12-\i)--(\i, 12-\i-3)--
		    (\i+3,12-\i-3)--(\i+3,12-\i)--cycle;
		    \draw[fill = black] (\i+2, 12-\i-1) circle (12pt);
		    \draw[fill = black] (\i+1, 12-\i-2) circle (12pt);
		    }
  	      \node [rotate=-45] at (-0.25,12.25) {{\footnotesize $\dots$}};
	  \end{tikzpicture}
&\quad\quad&
	  \begin{tikzpicture}[scale=0.125]
		  % draws the boundaries
		    \draw[thick, color=lightgray, line cap=round] (6,0)--(6,12)--(12,12)--(12,0)--(0,0)--(0,12)--(6,12);
		
		  % draws the permutation
		  \foreach \y [count = \x] in {2,4,6,8}
		    \draw[fill = black] (\x, 1.1*\y) circle (12pt);
		  \foreach \y [count = \x] in {1,3,5,7}
		    \draw[fill = black] (\x + 6, 1.1*\y) circle (12pt);
		
		  \node [rotate=60] at (5.125,10.75) {{\footnotesize $\dots$}};
		  \node [rotate=60] at (11.125,9.625) {{\footnotesize $\dots$}};
	  \end{tikzpicture}
&\quad\quad&
	  \begin{tikzpicture}[scale=0.125, rotate=-90]
		  % draws the boundaries
		    \draw[thick, color=lightgray, line cap=round] (6,0)--(6,12)--(12,12)--(12,0)--(0,0)--(0,12)--(6,12);
		
		  % draws the permutation
		  \foreach \y [count = \x] in {2,4,6,8}
		    \draw[fill = black] (\x, 1.1*\y) circle (12pt);
		  \foreach \y [count = \x] in {1,3,5,7}
		    \draw[fill = black] (12-\x, 1.1*\y) circle (12pt);
		  \node [rotate=30] at (6.75,9.75) {{\footnotesize $\dots$}};
		  \node [rotate=-30] at (5.125,10.75) {{\footnotesize $\dots$}};
	  \end{tikzpicture}
\end{tabular}
\end{center}
\caption{The two permutations shown on the left are the obstructions which prevent a class from being ``monotone griddable''. The two permutations on the right (and all of their symmetries) are the obstructions which prevent a monotone griddable class from being a polynomial class.}
\label{fig-mono-grid-obstructions}
\end{figure}

\begin{theorem}[The combination of the results of~{\cite{huczynska:grid-classes-an:}} and {\cite[Theorem 10.3]{albert:geometric-grid-:}}]
\label{thm-poly-tfae}
For a permutation class $\C$ the following are equivalent:
\begin{enumerate}
\item[(1)] $|\C_n|$ is given by a polynomial for all sufficiently large $n$,
\item[(2)] $|\C_n|<F_n$ for some $n$,
\item[(3)] $\C$ does not contain arbitrary long permutations of any of the forms 
shown in Figure~\ref{fig-mono-grid-obstructions} (or any symmetries of those), and
\item[(4)] $\C=\Grid(\tilde{G})$ for a finite set $\tilde{G}$ of peg 
permutations.
\end{enumerate}
\end{theorem}
\begin{proofsketch}
Huczynska and Vatter~\cite[Corollary 3.4]{huczynska:grid-classes-an:} prove that (1) and (2) are equivalent. They further prove \cite[Theorem 2.5 and Proposition 3.3]{huczynska:grid-classes-an:} that (3) implies that $\C\subseteq\Grid(\tilde{\rho})$ for some peg permutation $\tilde{\rho}$ (although in \cite{huczynska:grid-classes-an:} this is stated in an equivalent manner in terms of grid classes of matchings). This condition implies (1) by \cite[Theorem 2.9]{huczynska:grid-classes-an:}, and the converse---that (1) implies (3)---follows by an elementary counting argument.

It remains only to establish that (1), (2), and (3) are equivalent to (4). It follows readily that (4) implies (3). Finally, the work of Albert, Atkinson, Bouvel, Ru\v{s}kuc, and Vatter~\cite{albert:geometric-grid-:}, specifically Theorem 10.3, on ``atomic'' geometric grid classes shows that if $\C\subseteq\Grid(\tilde{\rho})$ for some peg permutation $\tilde{\rho}$ (as implied by any of (1), (2), or (3) by the above) then $\C=\Grid(\tilde{G})$ for a finite set $\tilde{G}$ of peg permutations.
\end{proofsketch}

Because we are only inflating peg permutations by monotone intervals, we can specify these intervals by vectors of positive integers rather than permutations. For example, using this notation we can write
\[
	6\ 321\ 7\ 45 = \d{3}\m{1}\d{4}\p{2}[\vect{1,3,1,2}].
\]

We denote the non-negative integers by $\mathbb{N}$ and the positive integers by $\mathbb{P}$. Thus $\mathbb{N}^m$ (resp., $\mathbb{P}^m$) denotes vectors of length $m$ with entries from $\mathbb{N}$ (resp., $\mathbb{P}$). For a vector $\vec{v}$ in one of these sets we write $\|\vec{v}\|=\sum \vec{v}(i)$ and refer to this quantity as the \emph{weight} of $\vec{v}$.

Let $\tilde{\rho}$ be a peg permutation of length $m$ and $\vec{v}\in\mathbb{N}^m$. If $\tilde{\rho}(i)$ is dotted, we must have $\vec{v}(i)\le 1$. Thus if $\tilde{\rho}$ is of length $m$, we can write
\[
	\Grid(\tilde{\rho})
	=
	\{\tilde{\rho}[\vec{v}] \st \vec{v}\in\mathbb{N}^m \mbox{ which satisfy } \vec{v}(i)\le 1 \mbox{ for all $i$ such that $\tilde{\rho}(i)$ is dotted}\}.
\]
Indeed, we impose a stronger constraint herein. In our theorem and algorithm, we insist on inflating $\tilde{\rho}$ by vectors which \emph{fill} them; this means that each component of the vector equals $1$ if it corresponds to a dotted entry of $\tilde{\rho}$ and is otherwise at least $2$. Given a set $\mathcal{V}\subseteq\mathbb{P}^m$ of vectors which fill $\tilde{\rho}$ we define
\[
	\tilde{\rho}[\mathcal{V}]
	=
	\{\tilde{\rho}[\vec{v}] \st \vec{v}\in\mathcal{V}\}.
\]

We also extend the notion of containment and avoidance to vectors. Given the vectors $\vec{v}$ and $\vec{w}$ in $\mathbb{N}^m$ or $\mathbb{P}^m$, we say that $\vec{v}$ is \emph{contained} in $\vec{w}$ if $\vec{v}(i)\le\vec{w}(i)$ for all indices $i$ (and write $\vec{v}\le\vec{w}$ in this case). We further say that $\vec{w}$ \emph{avoids} $\vec{v}$ if $\vec{v}$ is not contained in $\vec{w}$. The containment relation on $\mathbb{N}^m$ (and thus also on $\mathbb{P}^m$) is clearly a partial order and is compatible with permutation containment in the sense that if $\vec{v}\le\vec{w}$ then $\tilde{\rho}[\vec{v}]\le\tilde{\rho}[\vec{w}]$, assuming both inflations are defined. 

Because our order on vectors is a partial order we may define \emph{downsets} (sets closed downward under containment) and \emph{upsets} of vectors. The intersection of a downset and an upset is referred to as a \emph{convex set}. As with permutation classes, we can specify a downset of vectors by its basis, which consists of the minimal permutations not in the downset. Unlike the permutation class context, however, for vectors we are guaranteed by Higman's Theorem~\cite{higman:ordering-by-div:} that bases of downsets are finite. As every convex set is simply the set difference of two downsets, this implies that all convex sets of vectors can be specified by a finite amount of information. The fundamental objects in our algorithm are ordered pairs of the form $(\tilde{\rho}, \mathcal{V}_{\tilde{\rho}})$ where $\tilde{\rho}$ is a peg permutation of length $m$ and $\mathcal{V}_{\tilde{\rho}}$ is a convex set in $\mathbb{P}^m$.

Note that the set of vectors which fill a given peg permutation $\tilde{\rho}$ forms a convex set. The downset component of this convex set consists of those vectors which do not contain an entry larger than $1$ corresponding to a dotted entry of $\tilde{\rho}$. The upset component consists of those vectors which contain the \emph{minimal filling vector} of $\tilde{\rho}$, which is the vector $\vec{m}$ defined by $\vec{m}(i)=1$ if $\tilde{\rho}(i)$ is dotted and $\vec{m}(i)=2$ if $\tilde{\rho}(i)$ is signed (in fact, only \emph{principal} upsets of the form $\{\vec{v}\st \vec{v}\ge\vec{m}\}$ arise in our work).

We now have all the terminology and notation to state our structure theorem.

\begin{theorem}
\label{thm-polynomial-main}
For every polynomial permutation class $\C$ there is a finite set $\tilde{H}$ of peg permutations, each associated with its own convex set $\mathcal{V}_{\tilde{\rho}}$ of filling vectors, such that $\C$ can be written as the disjoint union
\[
	\C=\biguplus_{\tilde{\rho}\in\tilde{H}} \tilde{\rho}[\mathcal{V}_{\tilde{\rho}}].
\]	
\end{theorem}

We prove Theorem~\ref{thm-polynomial-main} in the next section by giving an algorithm to compute the set $\tilde{G}$ and the associated convex sets $\mathcal{V}_{\tilde{\rho}}$ for each $\tilde{\rho}\in\tilde{G}$. Once these objects are computed, the enumeration of the class is reduced to the enumeration of a finite number of convex sets of vectors, which is straight-forward. 
Further, this disjoint union allows efficient algorithms for both the generation of permutations in the class and for testing class membership.
In Section~\ref{sec-genome-rearrangement} we apply our approach to the study of genome rearrangement.

Before that, we should mention that there are several established approaches which could, \emph{theoretically}, be used to enumerate polynomial classes, but they each have drawbacks. 
\begin{itemize}
\item Polynomial classes are contained in geometric grid classes (see Theorem~\ref{thm-poly-tfae}), so they fall under the purview of the results of Albert, Atkinson, Bouvel, Ru\v{s}kuc, and Vatter~\cite{albert:geometric-grid-:}. However, the proofs of these results are nonconstructive. Indeed, our work can be viewed as illuminating some preliminary obstacles which an algorithmic approach to geometric grid classes would have to overcome.
\item Polynomial classes can be shown to contain only finitely many ``simple permutations'' (this follows from Theorem~\ref{thm-poly-tfae}), so the methods of Albert and Atkinson~\cite{albert:simple-permutat:} (or the refinements introduced by Brignall, Huczynska, and Vatter~\cite{brignall:simple-permutat:alg:}) could be used to compute their generating functions. While some steps toward implementing this approach have been taken by Bassino, Bouvel, Pierrot, Pivoteau, and Rossin~\cite{bassino:combinatorial-s:}, applying it would require us to first determine the basis of the class in question.
\item Polynomial classes can be enumerated using the insertion encoding of Albert, Linton, and Ru\v{s}kuc~\cite{albert:the-insertion-e:} (which is implemented in the Maple package {\sc InsEnc} described in Vatter~\cite{vatter:finding-regular:}). However, this method also requires determining the basis of the class.
\end{itemize}

\section{The Algorithm}

Presented with a set $\tilde{G}$ of peg permutations, the algorithm we describe outputs a set of peg permutations and corresponding convex set of integer vectors as specified in the statement of Theorem~\ref{thm-polynomial-main}. This set is then used to compute the generating function for the class $\Grid(\tilde{G})$ and determine the enumerating polynomial. This process is divided into five steps, each described in its own subsection. (Note that in the implementation, steps 3 and 4 are combined.)

\begin{enumerate}[1.]
\item The input set is first \emph{completed} by adding new peg permutations which are contained in the given set in a certain sense. 
\item The resulting set is then \emph{compacted}, by removing extraneous peg permutations. 
\item The set of peg permutations is then \emph{cleaned} and transformed into a set of pairs of peg permutations and convex sets of vectors which describe restrictions on their fillings. 
\item The resulting set is \emph{combined}, which expresses the polynomial class as a disjoint union of inflations of peg permutations by convex sets of vectors, as promised by Theorem~\ref{thm-polynomial-main}.
\item Finally, with this preprocessing accomplished, the generating function (and the polynomial) enumerating the class can be easily computed.
\end{enumerate}

Alongside our description of these steps we consider the example of enumerating $\Grid(\m 1\p 2)$. While this class is trivial to enumerate using more traditional methods, it serves to illustrate the various steps of the algorithm.

We need a few prerequisites before the algorithm can be described. First we define a partial order on peg permutations. Given peg permutations $\tilde{\tau}$ and $\tilde{\rho}$ of lengths $k$ and $n$, respectively, $\tilde{\tau}\le\tilde{\rho}$ if there are indices $1\le i_1<i_2<\cdots<i_k\le n$ such that $\rho(i_1)\rho(i_2)\cdots\rho(i_k)$ is order isomorphic to $\tau$ and for each $j$, $\tilde{\tau}(j)$ is decorated with a
\[
\left\{
\begin{array}{cl}
\text{$+$ or $\bullet$}&\text{if $\tilde{\rho}(i_j)$ is decorated with a $+$,}\\
\text{$-$ or $\bullet$}&\text{if $\tilde{\rho}(i_j)$ is decorated with a $-$, 
or}\\
\text{$\bullet$}&\text{if $\tilde{\rho}(i_j)$ is dotted.}\\
\end{array}
\right.
\]
In other words, in order to obtain a smaller element in this \emph{peg permutation order}, one can change signs to dots and delete entries. Note that $\Grid(\tilde{\tau})\subseteq\Grid(\tilde{\rho})$ whenever $\tilde{\tau}\le\tilde{\rho}$, but the reverse implication does not hold in general; for example, $\Grid(\d{1}\d{2})\subseteq\Grid(\p{1})$, but $\d{1}\d{2} \not\le \p{1}$.

In addition, we extend the notion of intervals to peg permutations in the trivial way, by ignoring decoration; thus the intervals of $\tilde{\rho}$ are the same as the intervals of $\rho$, although they carry their decoration from $\tilde{\rho}$. We must also say something about monotone intervals of (unpegged) permutations:

\begin{proposition}\label{prop-mono-intervals-intersection}
If two monotone intervals of a permutation intersect then their union is also a monotone interval.
\end{proposition}
\begin{proof}
Suppose that two monotone intervals of the permutation $\pi$ intersect. The claim is obvious if both are increasing, both are decreasing, or either consists of a single entry. Thus we may assume that both intervals have at least two entries, one is increasing, and the other is decreasing. In this case the decreasing interval must share either the first or last entry with the increasing interval, but both cases lead to contradictions, completing the proof.
\end{proof}

\subsection{Completion}

We say that the set $\tilde{G}$ of peg permutations is \emph{complete} if every $\pi\in\Grid(\tilde{G})$ fills some $\tilde{\rho}\in\tilde{G}$. It is not difficult to construct complete sets of peg permutations, as we observe below.

\begin{proposition}
\label{prop-CTI-complete}
Every downset in the peg permutation order is complete.
\end{proposition}
\begin{proof}
Suppose that $\pi\in\Grid(\tilde{G})$ for a downset $\tilde{G}$ of peg permutations. Therefore $\pi\in\Grid(\tilde{\rho})$ for some peg permutation $\tilde{\rho}\in\tilde{G}$, i.e., $\pi=\tilde{\rho}[\vec{v}]$ for a vector $\vec{v}\in\mathbb{N}^m$, where $m$ denotes the length of $\tilde{\rho}$. Now form the permutation $\tilde{\tau}$ by deleting the entries $\tilde{\rho}(i)$ for which $\vec{v}(i)=0$ and dotting the entries $\tilde{\rho}(i)$ for which $\vec{v}(i)=1$. Clearly $\tilde{\tau}\le\tilde{\rho}$, so $\tilde{\tau}\in\tilde{G}$ because it is a downset, and $\pi$ fills $\tilde{\tau}$, as desired.
\end{proof}

Let $\tilde{G}$ be an input set of peg permutations. The completion step consists of replacing $\tilde{G}$ by its downward closure, say $\tilde{H}$, which can be obtained by deleting entries and changing signs to dots for each of the peg permutations within $\tilde{G}$. Having done this, $\tilde{H}$ is a complete set of peg permutations, and so every permutation in the class $\Grid(\tilde{H})$ fills some member of $\tilde{H}$.

In our example $\tilde{G} = \{\m 1 \p 2\}$, so the first step in the algorithm is to replace it with its downward closure under the peg permutation order,
$$
\{\d 1, \d 1 \d 2, \p 1, \m 1, \d 1 \p 2, \m 1 \d 2, \m 1 \p 2\}.
$$
At this stage each peg permutation is associated to the convex set of vectors which fill it.

\subsection{Compacting}

Our next step is more technical, and requires additional definitions. Roughly, this step removes those elements of $\tilde{G}$ which define a grid class also defined by another member of $\tilde{G}$. 

Proposition~\ref{prop-mono-intervals-intersection} shows that every permutation $\pi$ has a unique coarsest partition into monotone intervals. In other words, for each $\pi$ there is a unique peg permutation $\tilde{\rho}$ such that $\pi$ is $\tilde{\rho}$-griddable, but not $\tilde{\tau}$-griddable for any $\tilde{\tau}<\tilde{\rho}$ (here $<$ denotes the peg permutation order). In particular, this implies that, for this $\tilde{\rho}$, $\Grid(\tilde{\rho})$ properly contains $\Grid(\tilde{\tau})$ for all peg permutations $\tilde{\tau} < \tilde{\rho}$. We call a peg permutation $\tilde{\rho}$ with this property \emph{compact}, i.e., $\tilde{\rho}$ is compact if $\Grid(\tilde{\tau})\subsetneq\Grid(\tilde{\rho})$ for all $\tilde{\tau}<\tilde{\rho}$. For example, $\d{2}\m{1}$ is not compact because $\m{1}<\m{2}\d{1}$ and $\Grid(\m{2}\d{1})=\Grid(\m{1})$, but both $\d1\d2$ and $\p1$ are compact.

Our next result ties the definitions of compactness and filling together.

\begin{proposition}
\label{prop-tfae-compact-filling}
For a peg permutation $\tilde{\rho}$, the following conditions are equivalent:
\begin{enumerate}
\item[(1)] $\tilde{\rho}$ is compact,
\item[(2)] $\tilde{\rho}$ does not have an interval order isomorphic to  
$\p{1}\p{2}$, $\p{1}\d{2}$, $\d{1}\p{2}$, or symmetrically, to $\m{2}\m{1}$, 
$\m{2}\d{1}$, $\d{2}\m{1}$, and
\item[(3)] every permutation which fills $\tilde{\rho}$ has a unique 
$\tilde{\rho}$-partition.
\end{enumerate}
\end{proposition}
\begin{proof}
It is clear that (1) implies (2), so our first task is to show that (2) implies (3). Suppose that $\tilde{\rho}$ satisfies the conditions of (2) but that there is a permutation $\pi$ which fills $\tilde{\rho}$ and has two different $\tilde{\rho}$-partitions. Equivalently, this means that $\pi=\tilde{\rho}[\vec{v}]$ for a vector $\vec{v}$ of positive integers (the filling partition) and that there is a different vector, $\vec{w}$ of nonnegative integers also with $\pi=\tilde{\rho}[\vec{w}]$.

Let $j$ denote the first index such that $\vec{v}(j)\neq\vec{w}(j)$. There must be at least one more entry where $\vec{v}$ and $\vec{w}$ differ, so let $k>j$ denote the first index after $j$ such that $\vec{v}(k)\neq\vec{w}(k)$. Because $\tilde{\rho}[\vec{v}]$ and $\tilde{\rho}[\vec{w}]$ yield the same permutation despite the fact that $\tilde{\rho}(j)$ and $\tilde{\rho}(k)$ are inflated by monotone intervals of different lengths, these two entries must lie in a common monotone interval of $\tilde{\rho}$. Moreover, on of the entries of this interval of $\tilde{\rho}$ must be signed, because $\tilde{\rho}[\vec{v}]$ is a filling partition for $\pi$. However, this implies that $\tilde{\rho}$ contains one of the intervals listed in (2), a contradiction.

It remains to show that (3) implies (1). Let $\vec{m}$ denote the minimal filling vector of $\tilde{\rho}$ (defined by $\vec{m}(i)=1$ if $\tilde{\rho}(i)$ is dotted and $\vec{m}(i)=2$ otherwise). By the hypotheses of (3), $\pi=\tilde{\rho}[\vec{m}]$ has a unique $\tilde{\rho}$-partition. This shows that $\pi$ is not contained in $\Grid(\tilde{\tau})$ for any $\tilde{\tau}<\tilde{\rho}$, so $\tilde{\rho}$ is compact, completing the proof.
\end{proof}

We say that the set $\tilde{G}$ of peg permutations is \emph{compact} if every peg permutation it contains is compact. Note that if $\tilde{G}$ is a downset and $\tilde{\rho} \in \tilde{G}$ is not compact then there is a $\tilde{\tau} \in \tilde{G}$ such that $\Grid(\tilde{\tau}) = \Grid(\tilde{\rho})$. This implies the following result.

\begin{proposition}\label{prop-compact-complete}
Let $\tilde{G}$ be a downset of peg permutations and $\tilde{H}$ the result of removing all non-compact peg permutations from $\tilde{G}$. Then $\Grid(\tilde{G}) = \Grid(\tilde{H})$. 
\end{proposition}

Proposition~\ref{prop-tfae-compact-filling} provides a simple method for identifying non-compact elements of $\tilde{G}$. During this step of the algorithm, we simply inspect each element of $\tilde{G}$ and remove it if it contains an interval isomorphic to one of those listed in Proposition~\ref{prop-tfae-compact-filling}, leaving a compact (but still complete) set which we denote by $\tilde{H}$. 

Our running example contains only a single non-compact peg permutation after the completion step, which we remove in the compacting step from $\tilde{G}$:
$$
\tilde{H} = \{\d 1, \d 1 \d 2, \p 1, \m 1, \cancel{\d 1 \p 2}, 
              \m 1 \d 2, \m 1 \p 2\}.
$$
The associated convex sets do not change in this step and so each peg permutation in $\tilde{G}$ is still associated to the convex set of vectors which fill it.

\subsection{Cleaning}

The compactification step removes some redundancies, but there is still a subtle problem to be addressed. In fact, this problem cannot be resolved at the level of peg permutations and requires us to restrict the associated convex sets for the first time. We refer to this step as \emph{cleaning} the set of peg permutations. 

For $\tilde{\rho}\in\tilde{G}$, if $\pi$ fills $\tilde{\rho}$ then $\pi$ has a unique $\tilde{\rho}$-partition, but this does not necessarily imply that $\pi$ doesn't fill some other $\tilde{\tau}\in\tilde{G}$. For example, $2341$ fills both $\d{2}\d{3}\d{4}\d{1}$ and $\p{2}\d{1}$. To address this problem, we say that a compact peg permutation $\tilde{\rho}$ is \emph{clean} if $\Grid(\tilde{\rho})\not\subseteq\Grid(\tilde{\tau})$ for any \emph{shorter} peg permutation $\tilde{\tau}$. We say that the set $\tilde{G}$ of peg permutations is \emph{clean} if each of them is clean.

\begin{proposition}
\label{prop-clean-iff}
The compact peg permutation $\tilde{\rho}$ is clean if and only if it does not have an interval order isomorphic to $\d{1}\d{2}$ or $\d{2}\d{1}$.
\end{proposition}
\begin{proof}
If $\tilde{\rho}$ contains an interval order isomorphic to $\d{1}\d{2}$ or $\d{2}\d{1}$ then let $\tilde{\tau}$ denote the peg permutation obtained by contracting this interval to a single entry decorated with the appropriate sign; clearly $\Grid(\tilde{\rho})\subseteq\Grid(\tilde{\tau})$.

Otherwise suppose that $\Grid(\tilde{\rho})\subseteq\Grid(\tilde{\tau})$ where $\tilde{\tau}$ is shorter than $\tilde{\rho}$ and let $\pi$ be any permutation which fills $\tilde{\rho}$. In any $\tilde{\tau}$-partition of $\pi$, since $\tilde{\tau}$ is shorter than $\tilde{\rho}$, there must be some interval which intersects at least two parts of a $\tilde{\rho}$-partition. Since this interval is monotone, this implies that the union of two intervals of $\tilde{\rho}$ is monotone. Because $\tilde{\rho}$ is compact, we see from our previous proposition that $\tilde{\rho}$ must contain a $\d{1}\d{2}$ or $\d{2}\d{1}$ interval, as desired.
\end{proof}

Given a complete and compact set $\tilde{G}$, it is not in general possible to obtain a clean subset $\tilde{H}\subseteq\tilde{G}$ with $\Grid(\tilde{H})=\Grid(\tilde{G})$; to return to our previous example, if $\d{2}\d{3}\d{4}\d{1}\in\tilde{G}$ but $\p{2}\d{1}\notin\tilde{G}$ then we cannot simply remove $\d{2}\d{3}\d{4}\d{1}$ from $\tilde{G}$ as we would lose permutations in doing so (and we cannot counteract this by adding $\p{2}\d{1}$ to $\tilde{G}$ as then we would gain permutations).

Instead, given a complete and compact set $\tilde{G}$ of peg permutations (with associated convex sets), the cleaning step restricts the convex sets. Specifically, this step performs the following operation.
\begin{itemize}
\item If $\tilde{\rho}$ is clean, leave it alone. 
\item Otherwise, locate the maximal intervals of the form $\d 1 \d 2 \dots \d k$ (resp., $\d k \dots \d 2 \d 1$) for $k\ge 2$ within $\tilde{\rho}$ and contract them to $\p 1$ (resp., $\m 1$). Call this new peg permutation $\tilde{\tau}$. To $\tilde{\tau}$ we associate the convex set of vectors which fill $\tilde{\tau}$ and for each entry $i$, if the $i$th entry of $\tilde{\tau}$ is the result of contracting $k$ dotted entries, then the vectors of $\mathcal{V}_{\tilde{\tau}}$ may not have their $i$th components greater than $k$.
\end{itemize}
In particular, note that the process of cleaning preserves the properties of completeness and compactness.

As an example of this process, the unclean peg permutation $\p 6 \d 3 \d 4 \d 5 \d 2 \d 1$ (associated to the convex set of vectors which fill it) becomes the peg permutation $\p 3 \p 2 \m 1$ associated to convex set of vectors which fill it and have their second component at most $3$ and their third component at most $2$. 

In our running example, the compacted set has only a single unclean peg permutation, $\d 1 \d 2$. Thus in the cleaning step we replace this peg permutation with the peg permutation $\p 1$ associated with the convex set of vectors $\{\vect{2}\}$. Note that our set already contained the peg permutation $\p 1$ associated with the convex set of vectors which fill it, $\{\vect{i}\st i\ge 2\}$. This overlap is handled in the next step of the algorithm.

\subsection{Combination}

The cleaning step may result in multiple convex sets associated to each peg permutation. This is fixed in the combination step, using properties of the poset $\mathbb{P}^m$. 

Note that $\mathbb{P}^m$ forms a lattice, with meet and join given, respectively, by component-wise minimum and maximum:
\begin{eqnarray*}
\vec{v}\wedge\vec{w}&=&(\min\{v(1),w(1)\},\dots,\min\{v(m),w(m)\}),\\
\vec{v}\vee\vec{w}&=&(\max\{v(1),w(1)\},\dots,\max\{v(m),w(m)\}).
\end{eqnarray*}
It follows that computing the union and intersection of arbitrary vector posets is relatively simple, as our next result shows.

\begin{proposition}
\label{prop-vector-union-intersection}
If $\V,\W\subseteq\mathbb{P}^m$ be downsets with bases $B_\V$ and $B_\W$, respectively, then $\V \cap \W$ and $\V \cup \W$ are also downsets. Further, the basis of $\V \cap \W$ is given by the minimal elements of the set $B_\V \cup B_\W$, and the basis of $\V \cup \W$ is given by the minimal elements of the set $ \{\vec{v}\vee\vec{w}\st\vec{v}\in B_\V\mbox{ and }\vec{w}\in B_\W\}$.
\end{proposition}
\begin{proof}
It is clear that the basis of $\V\cap\W$ is the set of minimal elements in the union of the two bases. Computing bases for unions is less transparent. If $B_\V$ and $B_\W$ are both singletons, consisting of $\vec{v}$ and $\vec{w}$, respectively, say, then the basis of $\V\cup\W$ is $\vec{v}\vee\vec{w}$. Therefore we see that for general bases,
\begin{eqnarray*}
\V\cup\W
&=&
\left(\bigcap_{\vec{v}\in B_\V} \{\mbox{$\vec{v}$-avoiding vectors}\}\right)
\bigcup
\left(\bigcap_{\vec{w}\in B_\W} \{\mbox{$\vec{w}$-avoiding vectors}\}\right),
\\
&=&
\bigcap_{\substack{\vec{v}\in B_\V,\\\vec{w}\in B_\W}}
\{\mbox{$\vec{v}$-avoiding vectors}\}\cup \{\mbox{$\vec{w}$-avoiding vectors}\},
\\
&=&
\bigcap_{\substack{\vec{v}\in B_\V,\\\vec{w}\in B_\W}}
\{\mbox{$\vec{v}\vee\vec{w}$-avoiding vectors}\},
\end{eqnarray*}
as claimed.
\end{proof}

In the combination step, for every peg permutation which has multiple convex sets associated to it we simply compute the union of these convex sets (using Proposition~\ref{prop-vector-union-intersection}) and update our list of pairs (of peg permutations and convex sets). At the conclusion of this step we have a set of peg permutations, each associated to a unique convex set. Moreover, we have established that every permutation $\pi\in\Grid({\tilde{G}})$ fills a unique clean and compact peg permutation $\tilde{\rho}$. Thus there is a unique vector $\vec{v}\in\mathcal{V}_{\tilde{\rho}}$ such that $\pi=\tilde{\rho}[\vec{v}]$. This proves Theorem~\ref{thm-polynomial-main}, which states that the permutation class $\Grid(\tilde{G})$ is indeed in bijection with the disjoint union
\[
	\biguplus_{\tilde{\rho}\in\tilde{G}} \tilde{\rho}[\mathcal{V}_{\tilde{\rho}}].
\]

At the end of this preprocessing, our running example contains five peg permutations, $\d 1$, $\p 1$, $\m 1$, $\m 1 \d 2$, and $\m 1 \p 2$ and these peg permutations are associated with the following convex sets
\begin{eqnarray*}
	\mathcal{V}_{\d 1}
	&=&
	\{\vect{1}\},\\
	\mathcal{V}_{\p 1}
	&=&
	\mathbb{P}\setminus\{\vect{1}\},\\
	\mathcal{V}_{\m 1}
	&=&
	\mathbb{P}\setminus\{\vect{1}\},\\
	\mathcal{V}_{\m 1 \d 2}
	&=&
	(\mathbb{P}\times\{1\})\setminus \{\vect{1,1}\},\\
	\mathcal{V}_{\m 1 \p 2}
	&=&
	\mathbb{P}^2\setminus \left(\mathbb{P}\times\{1\}\cup\{1\}\times\mathbb{P}\right).
\end{eqnarray*}

\subsection{Enumeration}

\SetAlFnt{\footnotesize\tt}
\begin{algorithm}
  \DontPrintSemicolon
  \SetAlgoLined
  \KwIn{Set $\tilde{G}$ of peg permutations}
  \KwOut{A set $\tilde{H}$ of peg permutations, each associated with a convex set $\V_{\tilde{\rho}}$ of vectors so that $\Grid(\tilde{G})$ is the disjoint union $\displaystyle\biguplus_{\tilde{\rho}\in \tilde{H}} \tilde{\rho}[\V_{\tilde{\rho}}]$.}

  \tcp{Complete $\tilde{G}$}
  \For{$\tilde{\rho} \ \in\tilde{G}$}{
    Add to $\tilde{G}$ all peg permutations which are contained in $\tilde{\rho}$ in the peg permutation order\;
  }

  \tcp{Compact $\tilde{G}$}
  \For{$\tilde{\rho} \ \in\tilde{G}$}{
    \If{\ttfamily $\tilde{\rho}$ contains intervals of the form
    $\p{1}\p{2}, \d{1}\p{2}, \p{1}\d{2}$, or their symmetries}{
      Remove $\tilde{\rho}$ from $\tilde{G}$\;
    }
  }

  \tcp{Clean and combine $\tilde{G}$}
  Initialize the set $\tilde{H}$, which will contain pairs $(\tilde{\rho}, 
  \V_{\tilde{\rho}})$ of peg permutations associated with convex sets of vectors\;
  \For{$\tilde{\rho} \  \in $ $\tilde{G}$}{
    \eIf{\ttfamily $\tilde{\rho}$ contains intervals of the form $\d{1}\d{2}$ or 
    $\d{2}\d{1}$}{
      Let $\tilde{\gamma}$ denote the cleaned $\tilde{\rho}$ and define $\V$ to be the set of integer vectors for which
      $\{\tilde{\gamma}[\vec{v}] : \vec{v} \in \V\} = \{\tilde{\rho}[\vec{v}] \st \vec{v} \mbox{ fills } \tilde{\rho}\}$\;
    }
    {
    	Let $\tilde{\gamma}=\tilde{\rho}$ and $\V=\{\vec{v}\st \vec{v}\mbox{ fills } \tilde{\rho}\}$\;
    }
    \eIf{
    	\ttfamily $(\tilde{\gamma}, \W) \in \tilde{H}$ for some $\W$}{Replace the element $(\tilde{\gamma},\W)$ with $(\tilde{\gamma}, \W \cup \V)$\;
    }
    {
      Add $(\tilde{\gamma}, \V)$ to $\tilde{H}$\;
    }
  }
  \Return $\tilde{H}$
%  The class $\Grid(\tilde{G})$ is equal to the disjoint union $\displaystyle\biguplus_{\tilde{\rho}\in\tilde{G}} \tilde{\rho}[\mathcal{V}_{\tilde{\rho}}]$.\;
\caption{Summary of the algorithm.}
\label{fig:algo}
\end{algorithm}

The first four steps of the algorithm are outlined in pseudocode in Figure~\ref{fig:algo}.  After these steps have been completed, we have a description of our class in the form specified by Theorem~\ref{thm-polynomial-main}. This form allows translation between permutations and vectors, enabling efficient enumeration and generation of elements in the class. The computation of the generating function enumerating the class is straightforward, as we show here. 

First, for any vector $\vec{w}\in\mathbb{P}^m$, the generating function (by weight) for vectors $\vec{v}\in\mathbb{P}^m$ which satisfy $\vec{v}\ge\vec{w}$ is
\[
	\frac{x^{\|\vec{w}\|}}{(1-x)^m}.
\]
Next, suppose we wish to enumerate a downset $\mathcal{V}$ of vectors of length $m$ with (finite) basis $B$. The Principle of Inclusion-Exclusion shows that the generating function (again by weight) for vectors in $\mathcal{V}$ is
\[
	\sum_{B \subseteq B_\V} (-1)^{|B|} \frac{x^{\| \bigvee B \|}}{(1 - x)^m}.
\]
Because the complement of an upset is a downset, we can use this formula to compute the generating function for any convex set of vectors, as promised. The generating functions for each of the convex sets in our running example are, respectively,
\[
	x, \frac{x^2}{1-x}, \frac{x^2}{1-x}, \frac{x^3}{1-x}, \frac{x^4}{(1-x)^2}.
\]
Summing these gives the generating function for the class,
\[
	\sum_{n \geq 1} | \Grid_n(\m 1 \p 2)| x^n
	=
	\frac{x}{(1-x)^2}
	=
	\sum_{n\ge 1} nx^n.
\]

The algorithm returns the generating function for the input class, and there are two conversions one might like to perform. First, a simple coefficient extraction gives the polynomial which enumerates the class (for sufficiently large values of $n$). Second, as this polynomial is integer-valued, it is well-known that it has integer coefficients when expressed in the binomial coefficient basis, and this expression can be computed using a change of basis matrix.

\section{Genome Rearrangement}
\label{sec-genome-rearrangement}

While the enumeration of permutation classes is an interesting problem in its own right, polynomial classes have found application to other fields. In this section we apply our algorithm to a problem from genetics, which presents an opportunity to demonstrate the utility of our algorithm by enabling computations which were previously infeasible. Polynomial permutation classes have applications to the field of evolutionary biology, as surveyed by Fertin, Labarre, Rusu, Tannier, and Vialette~\cite{fertin:combinatorics-o:}, who provide tables of data relating to the problem of evolutionary distance. The data presented in this section, computed using this algorithm, represents a significant extension of their computations.

The genes in a chromosomal genome may be thought of as discrete blocks of DNA, and thus labeled from $1$ to $n$ along the genome. During the process of evolution, the genes in a genome of one species might be rearranged via one or several operations and then appear in a different order (given by a permutation $\pi$) in the genome of a different species. By studying the number of operations required to transform the identity permutation into $\pi$ we may therefore get an estimate of how many mutations occurred in the evolution of the second species from the first. There are several different operations of interest, which we briefly survey in what follows.

In all of these operations, the class of permutations which can be obtained in at most $k$ operations away from the identity is a polynomial permutation class, and its structural description, as $\Grid(\tilde{G})$ for a set $\tilde{G}$ of peg permutations, is routine to compute. Thus using the Python package which implements the approach described in the previous section, we are able to automatically compute the polynomials enumerating these classes. The majority of these enumerations were not previously in the \OEISref. The new sequences are those numbered \OEISlink{A228392}--\OEISlink{A228401} and \OEISlink{A256181}. As observed by Kaiser and Klazar~\cite{kaiser:on-growth-rates:}, these polynomials have integer coefficients in the binomial coefficient basis, and we express them in this basis below.

All of the operations we survey are based on the notion of a \emph{block}, which is a contiguous sequence of entries. The block transposition operation was introduced by Bafner and Pevzner~\cite{bafna:sorting-by-tran:}. In a single \emph{block transposition} one is allowed to interchange two adjacent blocks of a permutation. Thus we may change
$$
\pi(1)\cdots\pi(i-1)\ \boxed{\pi(i)\cdots\pi(j-1)}\ \boxed{\pi(j)\cdots\pi(k-1)}\ \pi(k)\cdots\pi(n)
$$
into
$$
\pi(1)\cdots\pi(i-1)\ \boxed{\pi(j)\cdots\pi(k-1)}\ \boxed{\pi(i)\cdots\pi(j-1)}\ \pi(k)\cdots\pi(n).
$$

Plotting the permutation before and after a block transposition makes clear the connection to polynomial classes. The identity permutation of any length can be plotted as an increasing series of dots, which we represent as a straight line of positive slope. Transposing a horizontally contiguous subset of this line results in a grid of lines, as shown in Figure~\ref{fig-one-transpose}. 

\begin{figure}[t] \centering
  \begin{tikzpicture}[scale=.125]
    \draw [thick, line cap=round] (0,0) -- (3,3);
    \draw [thick, line cap=round] (3,3) -- (6,6);
    \draw [thick, line cap=round] (6,6) -- (9,9);
    \draw [thick, line cap=round] (9,9) -- (12,12);

    % this is to keep them all at same vertical level
    \node at (6,-2) {};

    \draw [darkgray, thick, line cap=round] (0,0) -- (12,0) -- (12,12) -- (0,12) -- cycle;
  \end{tikzpicture}
\quad\quad
  \begin{tikzpicture}[scale=.125]
    \draw (3,-.5) .. controls (3,-1) and (4.5,-.5) .. (4.5,-1);
    \draw (6,-.5) .. controls (6,-1) and (4.5,-.5) .. (4.5,-1);
    \draw[->] (4.5, -1) .. controls (4.5,-3) and (9,-3) .. (9, -.5);
    \node at (6,-2) {};

    \draw [darkgray, thick, fill=lightgray, line cap=round] (3,0) rectangle (6,12);

    \draw [thick, line cap=round] (0,0) -- (3,3);
    \draw [thick, line cap=round] (3,3) -- (6,6);
    \draw [thick, line cap=round] (6,6) -- (9,9);
    \draw [thick, line cap=round] (9,9) -- (12,12);

    \draw [darkgray, thick, line cap=round] (0,0) -- (12,0) -- (12,12) -- (0,12) -- cycle;
    \foreach \i in {3,6,9}{
      \draw [darkgray, thick, line cap=round] (0, \i) -- (12, \i);
      \draw [darkgray, thick, line cap=round] (\i, 0) -- (\i, 12);
    }
  \end{tikzpicture}
\quad\quad
  \begin{tikzpicture}[scale=.125]
    \draw [thick, line cap=round] (0,0) -- (3,3);
    \draw [thick, line cap=round] (3,6) -- (6,9);
    \draw [thick, line cap=round] (6,3) -- (9,6);
    \draw [thick, line cap=round] (9,9) -- (12,12);

    \node at (6,-2) {};

    \draw [darkgray, thick, line cap=round] (0,0) -- (12,0) -- (12,12) -- (0,12) -- cycle;
    \foreach \i in {3,6,9}{
      \draw [darkgray, thick, line cap=round] (0, \i) -- (12, \i);
      \draw [darkgray, thick, line cap=round] (\i, 0) -- (\i, 12);
    }
  \end{tikzpicture}
  \caption{Transposing a horizontally contiguous subset of entries 
           of the identity permutation, represented by the peg pattern $\p 1$, 
           results in the peg pattern $\p 1 \p 3 \p 2 \p 4$.}
  \label{fig-one-transpose}
\end{figure}
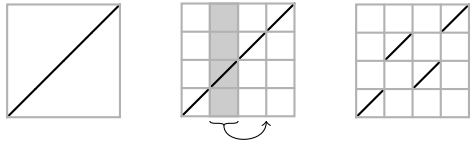

In the language of grid classes, the set of permutations which can be generated by a single block transposition from the identity is $\Grid(\p{1}\p{3}\p{2}\p{4})$. To compute the grid class of permutations which are at most two transpositions away from the identity, we simply repeat the operation on the figure again, in all possible ways. The result is a set of grid classes, the union of which is the desired set of permutations. Below we include the data for permutations which can be generated from the identity with $3$ or fewer block transpositions. Note that the polynomials given are only valid for sufficiently large $n$.
\begin{footnotesize}
$$
\begin{array}{r|rrrrrrrrrr|c}\hline
\diagnk&\ 1\ &\ 2\ &\ 3\ &\ 4\ &\ 5\ &\ 6\ &\ 7\ &\ 8\ &\ 9\ &10&\mbox{OEIS reference}\\[0.5ex]\hline
1&1&2&5&11&21&36&57&85&121&166&\OEISlink{A000292}\\[1ex]
&\multicolumn{10}{c|}{{n\choose 0}+{n\choose 2}+{n\choose 3}}&\\[1ex]
2&1&2&6&23&89&295&827&2017&4405&8812&\OEISlink{A228392}\\[1ex]
&\multicolumn{10}{c|}{{n\choose 0}+{n\choose 2}+2{n\choose 3}+8{n\choose 
4}+18{n\choose 5}+11{n\choose 6}}&\\[1ex]
3&1&2&6&24&120&675&3527&15484&56917&179719&\OEISlink{A228393}\\[1ex]
&\multicolumn{10}{c|}{\mbox{\scriptsize $\nc0 + \nc2 + 2\nc3 + 9\nc4 + 44\nc5 + 220\nc6 + 656\nc7 + 841\nc8 + 369\nc9$} }&\\[1ex]\hline
\end{array}
$$
\end{footnotesize}

There are a number of other well-studied block operations which model genome rearrangement. A \emph{prefix block transposition} is a special case of a block transposition in which the blocks must be at the beginning of the permutation. This method of rearrangement was first studied by Dias and Meidanis~\cite{dias:sorting-by-pref:}. The data for permutations which can be generated from the identity by $3$ or fewer prefix block transpositions is below; again, the polynomials are only valid for sufficiently large $n$.
\begin{footnotesize}
$$
\begin{array}{r|rrrrrrrrrr|c}\hline
\diagnk&\ 1\ &\ 2\ &\ 3\ &\ 4\ &\ 5\ &\ 6\ &\ 7\ &\ 8\ &\ 9\ &10&\mbox{OEIS reference}\\[0.5ex]\hline
1& 1& 2& 4& 7& 11& 16& 22& 29& 37& 46& \OEISlink{A000124}\\[1ex]
&\multicolumn{10}{c|}{\nc0 + \nc2  }&\\[1ex]
2& 1& 2& 6& 21& 61& 146& 302& 561& 961& 1546& \OEISlink{A228394}\\[1ex]
&\multicolumn{10}{c|}{\nc0 + \nc2  + 2\nc3 + 6\nc4 }&\\[1ex]
3& 1& 2& 6& 24& 116& 521& 1877& 5531& 13939& 31156& \OEISlink{A228395}\\[1ex]
&\multicolumn{10}{c|}{\nc0 + \nc2 + 2\nc3 + 9\nc4 + 40\nc5 + 90\nc6  }&\\[1ex]\hline
\end{array}
$$
\end{footnotesize}

A \emph{reversal} reverses one block in a permutation, thus transforming
$$
\pi(1)\cdots\pi(i-1)\ \boxed{\pi(i)\cdots\pi(j-1)}\ \pi(j)\cdots\pi(n)
$$
into
$$
\pi(1)\cdots\pi(i-1)\ \boxed{\pi(j-1)\cdots\pi(i)}\ \pi(j)\cdots\pi(n).
$$
Hence the class of permutations which can be generated by a single reversal is $\Grid(\p{1}\m{2}\p{3})$ (see Figure~\ref{fig-reversals-two-moves}) . This method of rearrangement was first introduced by Watterson, Ewens, Hall, and Morgan~\cite{watterson:the-chromosome-:}. Below is our data for this operation.

\begin{figure}[t] \centering
  \begin{tikzpicture}[scale=.15625]
    \draw[thick, line cap = round] (0,0) -- (4,4);
    \draw[thick, line cap = round] (4,8) -- (8,4);
    \draw[thick, line cap = round] (8,8) -- (12,12);
    \draw[darkgray, thick, line cap = round] (0,0) -- (12,0) -- (12,12) -- (0,12) -- cycle;
    \foreach \i in {4,8}{
      \draw[darkgray, thick, line cap = round] (0, \i) -- (12, \i);
      \draw[darkgray, thick, line cap = round] (\i, 0) -- (\i, 12);
    }
  \end{tikzpicture}
  \hspace{1pc}
  \begin{tikzpicture}[scale=.125]
    \draw[thick, line cap = round] (0,0) -- (3,3);
    \draw[thick, line cap = round] (12,12) -- (15,15);
    \draw[thick, line cap = round] (3, 12) -- (6,9);
    \draw[thick, line cap = round] (6,6) -- (9,9);
    \draw[thick, line cap = round] (9,6) -- (12,3);
    \draw[darkgray, thick, line cap = round] (0,0) -- (15,0) -- (15,15) -- (0,15) -- cycle;
    \foreach \i in {3,6,9,12}{
      \draw[darkgray, thick, line cap = round] (0, \i) -- (15, \i);
      \draw[darkgray, thick, line cap = round] (\i, 0) -- (\i, 15);
    }
  \end{tikzpicture}
  \hspace{1pc}
  \begin{tikzpicture}[scale=.125]
    \draw[thick, line cap = round] (0,0) -- (3,3);
    \draw[thick, line cap = round] (12,12) -- (15,15);
    \draw[thick, line cap = round] (3, 6) -- (6,3);
    \draw[thick, line cap = round] (6,6) -- (9,9);
    \draw[thick, line cap = round] (9,12) -- (12,9);
    \draw[darkgray, thick, line cap = round] (0,0) -- (15,0) -- (15,15) -- (0,15) -- cycle;
    \foreach \i in {3,6,9,12}{
      \draw[darkgray, thick, line cap = round] (0, \i) -- (15, \i);
      \draw[darkgray, thick, line cap = round] (\i, 0) -- (\i, 15);
    }
  \end{tikzpicture}
  \hspace{1pc}
  \begin{tikzpicture}[scale=.125]
    \draw[thick, line cap = round] (0,0) -- (3,3);
    \draw[thick, line cap = round] (12,12) -- (15,15);
    \draw[thick, line cap = round] (3, 9) -- (6,12);
    \draw[thick, line cap = round] (6,6) -- (9,3);
    \draw[thick, line cap = round] (9,9) -- (12,6);
    \draw[darkgray, thick, line cap = round] (0,0) -- (15,0) -- (15,15) -- (0,15) -- cycle;
    \foreach \i in {3,6,9,12}{
      \draw[darkgray, thick, line cap = round] (0, \i) -- (15, \i);
      \draw[darkgray, thick, line cap = round] (\i, 0) -- (\i, 15);
    }
  \end{tikzpicture}
  \hspace{1pc}
  \begin{tikzpicture}[scale=.125]
    \draw[thick, line cap = round] (0,0) -- (3,3);
    \draw[thick, line cap = round] (12,12) -- (15,15);
    \draw[thick, line cap = round] (3,9) -- (6,6);
    \draw[thick, line cap = round] (6,12) -- (9,9);
    \draw[thick, line cap = round] (9,3) -- (12,6);
    \draw[darkgray, thick, line cap = round] (0,0) -- (15,0) -- (15,15) -- (0,15) -- cycle;
    \foreach \i in {3,6,9,12}{
      \draw[darkgray, thick, line cap = round] (0, \i) -- (15, \i);
      \draw[darkgray, thick, line cap = round] (\i, 0) -- (\i, 15);
    }
  \end{tikzpicture}
  \caption{The class of permutations which are at most one reversal away from the identity is the grid class shown on the left. The class of permutations which are at most two reversals away from the identity is the union of the second through fifth grid classes represented.}
  \label{fig-reversals-two-moves}
\end{figure}
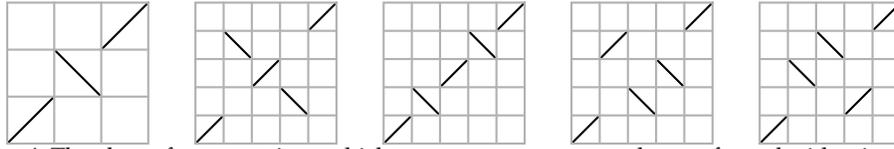

\begin{footnotesize}
$$
\begin{array}{r|rrrrrrrrrr|c}\hline
\diagnk&\ 1\ &\ 2\ &\ 3\ &\ 4\ &\ 5\ &\ 6\ &\ 7\ &\ 8\ &\ 9\ &10&\mbox{OEIS reference}\\[0.5ex]\hline
1&1&2&4&7&11&16&22&29&37&46&\OEISlink{A000124}\\[1ex]
&\multicolumn{10}{c|}{{n\choose 0}+{n\choose 2}}&\\[1ex]
2&1&2&6&22&63&145&288&516&857&1343&\OEISlink{A228396}\\[1ex]
&\multicolumn{10}{c|}{8{n\choose 0}-3{n\choose 1}+{n\choose 2}+4{n\choose 3}}&\\[1ex]
3&1&2&6&24&118&534&1851&5158&12264&25943&\OEISlink{A228397}\\[1ex]
&\multicolumn{10}{c|}{\mbox{ \scriptsize $318\nc0 -214\nc1 +131\nc2 -61\nc3 +20\nc4 +70\nc5 +35\nc6$  }}&\\[1ex]\hline
\end{array}
$$
\end{footnotesize}

By restricting reversals to initial segments of a permutation we obtain the \emph{prefix reversal} operation, which was introduced under the name \emph{pancake sorting} by ``Harry Dweighter'' (actually, Jacob E. Goodman) as a \emph{Monthly} problem~\cite{dweighter:elementary-prob:}. The grid class $\Grid(\m 1 \p 2)$, the example studied in the Section 2, is class of permutations which are one prefix reversal away from the identity.
\begin{footnotesize}
$$
\begin{array}{r|rrrrrrrrrr|c}\hline
\diagnk&1&2&3&4&5&6&7&8&9&10&\mbox{OEIS reference}\\[0.5ex]\hline
1&1&2&3&4&5&6&7&8&9&10&\OEISlink{A000027}\\[1ex]
&\multicolumn{10}{c|}{{n\choose 1}}&\\[1ex]
2&1&2&5&10&17&26&37&50&65&82&\OEISlink{A002522}\\[1ex]
&\multicolumn{10}{c|}{2 \nc0 -1 \nc1 + 2\nc2  }&\\[1ex]
3&1&2&6&21&52&105&186&301&456&657&\OEISlink{A228398}\\[1ex]
&\multicolumn{10}{c|}{ -3 \nc0 + 3 \nc1 - 2\nc2 + 6\nc3}&\\[1ex]\hline
\end{array}
$$
\end{footnotesize}

The cut-and-paste operation is a generalization of both the reversal operation and the block transposition operation. A single \emph{cut-and-paste} move consists of moving a single block of the permutation anywhere else in the permutation, with the option of reversing this block at the same time. Cut-and-paste sorting was introduced by Cranston, Sudborough, and West~\cite{cranston:short-proofs-fo:}.
\begin{footnotesize}
$$
\begin{array}{r|rrrrrrrrrr|c}\hline
\diagnk&1&2&3&4&5&6&7&8&9&10&\mbox{OEIS reference}\\[0.5ex]\hline
1&1&2&6&16&35&66&112&176&261&370&\OEISlink{A060354}\\[1ex]
&\multicolumn{10}{c|}{\nc1 + 3\nc3 }&\\[1ex]
2&1&2&6&24&120&577&2208&6768&17469&39603&\OEISlink{A228399}\\[1ex]
&\multicolumn{10}{c|}{-18\nc0 + 45\nc1 - 61\nc2 + 70\nc3 - 53\nc4 + 88\nc5 + 107\nc6}&\\[1ex]
3&1&2&6&24&120&720&5040&36757&223898&1055479&\OEISlink{A228400}\\[1ex]
&\multicolumn{10}{c|}{508264\nc0 - 280036\nc1 + 140012\nc2 - 57622\nc3 + 13839\nc4}&\\[1ex]
&\multicolumn{10}{c|}{+ 4136\nc5-5368\nc6 + 531\nc7 + 21125\nc8 + 12615\nc9}\\[1ex]
\hline
\end{array}
$$
\end{footnotesize}

Finally, the \emph{block interchange} operation is similar to the block transposition operation except that in this operation we are allowed to interchange any two blocks. This operation was first studied by Christie~\cite{christie:sorting-permuta:}.
\begin{footnotesize}
$$
\begin{array}{r|rrrrrrrrrr|c}\hline
\diagnk&1&2&3&4&5&6&7&8&9&10&\mbox{OEIS reference}\\[0.5ex]\hline
1&1&2&6&16&36&71&127&211&331&496&\OEISlink{A145126}\\[1ex]
&\multicolumn{10}{c|}{\nc0+\nc2+2\nc3+\nc4}&\\[1ex]
2&1&2&6&24&120&540&1996&6196&16732&40459&\OEISlink{A228401}\\[1ex]
&\multicolumn{10}{c|}{\nc0+\nc2+2\nc3+9\nc4+44\nc5+85\nc6+70\nc7+21\nc8}&\\[1ex]
3&1&2&6&24&120&720&5040&32256&169632&737364&\OEISlink{A256181}\\[1ex]
&\multicolumn{10}{c|}{\nc0+\nc2+2\nc3+9\nc4+44\nc5+265\nc6+1854\nc7+6769\nc8}\\[1ex]
&\multicolumn{10}{c|}{+12824\nc9+13125\nc{10}+6930\nc{11}+1485\nc{12}}\\[1ex]
\hline
\end{array}
$$
\end{footnotesize}

% 3 block interchanges: -x*(x^12-12*x^11+66*x^10-217*x^9+567*x^8-270*x^7+1608*x^6-541*x^5+419*x^4-184*x^3+58*x^2-11*x+1)/(x^13-13*x^12+78*x^11-286*x^10+715*x^9-1287*x^8+1716*x^7-1716*x^6+1287*x^5-715*x^4+286*x^3-78*x^2+13*x-1)

The Python code used to perform these computations is available at
\begin{center}
\url{https://github.com/cheyneh/polypermclass}.
\end{center}

\bigskip
\noindent{\bf Acknowledgments:} We thank Michael Engen, Jay Pantone, and the anonymous referees for their many comments which improved the presentation of the paper. We are additionally grateful to Jay Pantone for performing the computations on permutations at most $3$ block interchanges away from the identity.

\bigskip

\bibliographystyle{acm}
\bibliography{../../refs}

\end{document}